\documentclass[a4paper,11pt]{article}

\usepackage[T1]{fontenc}
\usepackage{amsfonts}
\usepackage{amssymb}
\usepackage{amsmath}

\usepackage{graphicx}

\usepackage{latexsym}

\usepackage{verbatim}

\makeatletter
\renewcommand{\@evenfoot}{\hfil - \thepage\ - \hfil}
\renewcommand{\@oddfoot}{\hfil - \thepage\ - \hfil}
\makeatother

\newcounter{theo}

\newenvironment{thm}
{\noindent\hangafter=1\hangindent=15pt\refstepcounter{theo}\textsc{Theorem \thetheo. --}\begin{sffamily}}
{\end{sffamily}\par}

\newenvironment{thm*}
{\noindent\hangafter=1\hangindent=15pt\refstepcounter{theo}\textsc{Theorem --}\begin{sffamily}}
{\end{sffamily}\par}

\newenvironment{thmnn}
{\noindent\hangafter=1\hangindent=15pt\textsc{Theorem --}\begin{sffamily}}
{\end{sffamily}\par}

\newenvironment{thmt*}[1]
{\noindent\hangafter=1\hangindent=15pt\refstepcounter{theo}\textsc{Theorem -- #1}\\\begin{sffamily}}
{\end{sffamily}\par}

\newenvironment{prop}
{\noindent\hangafter=1\hangindent=15pt\refstepcounter{theo}\textsc{Proposition \thetheo. --}\begin{sffamily}}
{\end{sffamily}\par}

\newenvironment{prop*}
{\noindent\hangafter=1\hangindent=15pt\textsc{Proposition. --}\begin{sffamily}}
{\end{sffamily}\par}

\newenvironment{lemme}
{\noindent\hangafter=1\hangindent=15pt\refstepcounter{theo}\textsc{Lemma \thetheo. -- }\begin{sffamily}}
{\end{sffamily}\par}

\newenvironment{lemme*}
{\noindent\hangafter=1\hangindent=15pt\refstepcounter{theo}\textsc{Lemma. -- }\begin{sffamily}}
{\end{sffamily}\par}

\newenvironment{coroll}
{\noindent\hangafter=1\hangindent=15pt\refstepcounter{theo}\textsc{Corollary \thetheo. -- }\begin{sffamily}}
{\end{sffamily}\par}

\newenvironment{corollnn}
{\noindent\hangafter=1\hangindent=15pt\textsc{Corollary. -- }\begin{sffamily}}
{\end{sffamily}\par}

\newenvironment{definition}
{\noindent\hangafter=1\hangindent=15pt\textsc{Definition. -- }\begin{sffamily}}
{\end{sffamily}\par}

\newenvironment{proof}
{\noindent\textsc{Proof.}\begin{rmfamily}\setlength{\parskip}{2pt}\setlength{\parindent}{0pt}}
{\end{rmfamily}\hfill$\square$\par}

\newcommand{\bP}{\mathbb{P}}
\newcommand{\bE}{\mathbb{E}}
\newcommand{\ensemble}[2]{{\left.\left\{#1\,\right|\,#2\right\}}} 

\newcommand{\AS}{\mbox{-a.s.}}
\newcommand{\aE}[1]{E_{#1}} 
\newcommand{\aP}[1]{P_{#1}}
\newcommand{\aPo}{\aP{0}}

\newcommand{\qE}[1]{E_{{#1},\omega}} 
\newcommand{\qP}[1]{P_{{#1},\omega}}
\newcommand{\qPo}{\qP{0}}

\newcommand{\Ht}{\widetilde{H}} 
\newcommand{\Er}{\mathcal{E}} 

\newcommand{\eps}{\varepsilon}

\newcommand{\limite}[1]{\xrightarrow[#1]{}}

\newcommand{\R}{\mathbb{R}}

\newcommand{\Z}{\mathbb{Z}}

\newcommand{\ot}{{\widetilde{o}}} 

\newcommand{\eqd}{\begin{eqnarray*}} 
\newcommand{\eqf}{\end{eqnarray*}} 

\newcommand{\dir}{\mathcal{D}}

\begin{document}

\begin{center}
\Large Integrability of exit times and ballisticity for random walks in~Dirichlet~environment

\vspace{.5cm}

\large
Laurent Tournier\footnote{Universit\'e de Lyon ; universit\'e Lyon 1 ; Institut Camille Jordan CNRS UMR 5208 ; 43, boulevard du 11 novembre 1918, F-69622 Villeurbanne Cedex. \emph{E-mail : }{\tt tournier@math.univ-lyon1.fr}}

\vspace{1cm}

\begin{minipage}{16cm}
\footnotesize
We prove an equivalent condition of integrability for Green functions and the exit time of random walks in random Dirichlet environment on finite digraphs, and apply this result to improve the ballisticity criterion given by Enriquez and Sabot in~\cite{EnriquezSabot06}. 
\end{minipage}
\end{center}

\section{Introduction}

Since their introduction in the 70's, models of random walks in random environment have mostly been studied in the one dimensional case. Using specific features of this setting, like the reversibility of the Markov chain, Solomon~\cite{Solomon75} set a first milestone by proving simple explicit necessary and sufficient conditions for transience, and a law of large numbers. In contrast, the multidimensional situation is still poorly understood. A first general transience criterion was provided by Kalikow~\cite{Kalikow}, which Sznitman and Zerner~\cite{SznitmanZerner} later proved to imply ballisticity as well. Under an additional uniform ellipticity hypothesis, Sznitman (\cite{Sznitman01}, \cite{Sznitman02}) could weaken this ballisticity criterion, but not much progress was made since then about the delicate question of sharpening transience or ballisticity criterions.

Another approach consists in deriving explicit conditions in more specific random environments. Among them, Dirichlet environments, first studied by Enriquez and Sabot in~\cite{EnriquezSabot06}, appear as a natural choice because of their connection with oriented edge linearly reinforced random walks (cf.~\cite{EnriquezSabot02}). Another interest in this case comes from the existence of algebraic relations involving Green functions. These allowed Enriquez and Sabot to show that Kalikow's criterion is satisfied under some simple condition, thus proving ballistic behaviour, and to give estimates of the limiting velocity. 

Defining Kalikow's criterion raises the problem of integrability of Green functions on finite subsets. While this property is very easily verified for a uniformly elliptic environment, it is no longer the case in the Dirichlet situation. In~\cite{EnriquezSabot06}, the condition on the environment allowed for a quick proof, and the general case remained unanswered. 

The main aim of this article is to state and prove a simple necessary and sufficient condition of integrability of these Green functions in Dirichlet environment on general directed graphs. Integrability conditions for exit times are then easily deduced. The "sufficiency" part of the proof is the more delicate. It procedes by induction on the size of the graph by going through an interesting quotienting procedure. 

This sharpening of the integrability criterion, along with an additional trick, allows us to prove a refined version of Enriquez and Sabot's ballisticity criterion. The condition of non integrability may also prove useful in further analysis of random walks in Dirichlet environment. Indeed, finite subsets with non integrable exit times play the role of "strong traps" for the walk. As a simple example, one can prove that the existence of such a subset implies a null limiting velocity. 

Next section introduces the notations, states the results and various corollaries. Section~\ref{sec:proof} contains the proofs of the main result and corollary. Finally, section~\ref{sec:bali} proves the generalization of Enriquez and Sabot's criterion. 

\section{Definitions and statement of the results}

\subsection{Dirichlet distribution}\label{sec:dir}

Let us first recall a few useful properties of the usual Dirichlet distribution. 
Let $I$ be a finite set. The set of probability distributions on $I$ is denoted by $\mathrm{Prob}(I)$ : 
$$\textstyle\mathrm{Prob}(I)=\ensemble{(p_i)_{i\in I} \in \R_+^I}{\sum_{i\in I} p_i =1}.$$
Given a family $(\alpha_i)_{i\in I}$ of positive real numbers, the {\bf Dirichlet distribution} of parameter $(\alpha_i)_{i\in I}$ is the probability distribution $\dir((\alpha_i)_{i\in I})$ on $\mathrm{Prob}(I)$ (the set of probability distributions on $I$) of density : 
\begin{equation}
(x_i)_{i\in I}\mapsto\frac{\Gamma(\sum_{i\in I}\alpha_i)}{\prod_{i\in I}\Gamma(\alpha_i)}\prod_{i\in I}x_i^{\alpha_i-1} \tag{$*$}
\end{equation}
with respect to the Lebesgue measure $\prod_{i\neq i_0} dx_i$ (where $i_0$ is any element of $I$) on the simplex $\mathrm{Prob}(I)$. Notice that if $(p_1,p_2)$ is a random variable sampled according to distribution $\dir(\alpha,\beta)$, then $p_1$ is a Beta variable of parameter $(\alpha,\beta)$. 
An easy computation shows that if $(p_1,\ldots,p_n)$ is a random variable sampled according to $\dir(\alpha_1,\ldots,\alpha_n)$ then, for $i=1,\ldots,n$, the expected value of $p_i$ is $\frac{\alpha_i}{\sum_{1\leq j\leq n} \alpha_j}.$

The following two important properties are simple consequences of the representation of a Dirichlet random variable as a renormalized vector of independent gamma random variables (cf. for instance~\cite{Wilks}). Let $(p_i)_{i\in I}$ be a random variable distributed according to $\dir((\alpha_i)_{i\in I})$. Then : 
\begin{description}
	\item[(Associativity)] Let $I_1,\ldots,I_n$ be a partition of $I$. The random variable $\left(\sum_{i\in I_k}p_i \right)_{k\in\{1,\ldots,n\}}$ on $\mathrm{Prob} (\{1,\ldots,n\})$ follows the Dirichlet distribution $\dir((\sum_{i\in I_k} \alpha_i)_{1\leq k\leq n})$. 
	\item[(Restriction)] Let $J$ be a nonempty subset of $I$. The random variable $\left(\frac{p_i}{\sum_{j\in J} p_j}\right)_{i\in J}$ on $\mathrm{Prob}(J)$ follows the Dirichlet distribution $\dir((\alpha_i)_{i\in J})$ and is independent of $\sum_{j\in J}p_j$ (which follows a Beta distribution $B(\sum_{j\in J}\alpha_j, \sum_{j\notin J} \alpha_j)$ due to the associativity property). 
\end{description}

\subsection{Definition of the model}

In order to deal with multiple edges, we define a {\bf directed graph} as a quadruplet $G=(V,E,\mathit{head},\mathit{tail})$ where $V$ and $E$ are two sets whose elements are respectively called the {\bf vertices} and {\bf edges} of $G$, endowed with two maps $\mathit{head}:e\mapsto\overline{e}$ and $\mathit{tail}:e\mapsto\underline{e}$ from $E$ to $V$. An edge $e\in E$ is thought of as an oriented link from $\underline{e}$ (\emph{tail}) to $\overline{e}$ (\emph{head}), and the usual definitions apply. Thus, a vertex $x$ is {\bf connected} to a vertex $y$ in $G$ if there is an oriented path from $x$ to $y$, i.e. a sequence $e_1,\ldots,e_n$ of edges with $\overline{e_k}=\underline{e_{k+1}}$ for $k=1,\ldots,n-1$, $\underline{e_1}=x$ and $\overline{e_n}=y$. For brevity, we usually only write $G=(V,E)$, the tail and head of an edge $e$ being always denoted by $\underline{e}$ and $\overline{e}$. 

In the following, we will usually deal with graphs $G=(V\cup\{\partial\},E)$ possessing a {\bf cemetery vertex}~$\partial$. In this case, we always suppose that :
\begin{enumerate}
	\item $\partial$ is a dead end: no edge in $E$ exits this vertex, and random walks remain stuck at this point once they have reached it ; 
	\item every vertex is connected to $\partial$. 
\end{enumerate}

Let $G=(V\cup\{\partial\},E)$ be such a graph. For all $x\in V$, let $\mathcal{P}_x$ designate the set of probability distributions on the set of edges originating at~$x$: $$\textstyle\mathcal{P}_x=\ensemble{(p_e)_{e\in E,\,\underline{e}=x}\in\R_+^\ensemble{e\in E}{\underline{e}=x}}{\sum_{e\in E,\,\underline{e}=x} p_e =1 }.$$
Then the set of {\bf environments} is $\Omega=\prod_{x\in V}\mathcal{P}_x\subset \R^E$. We will denote by $\omega=(\omega_e)_{e\in E}$ the canonical random variable on $\Omega$. 

Given a family $\vec{\alpha}=(\alpha_e)_{e\in E}$ of positive weights indexed by the set of edges of $G$, one can then define {\bf Dirichlet distribution on environments} of parameter $\vec{\alpha}$: this distribution on $\Omega=\prod_{x\in V}\mathcal{P}_x$ is the product measure of Dirichlet distributions on each of the $\mathcal{P}_x$, $x\in V$ : $$\bP=\bP^{(\vec{\alpha})} = \bigotimes_{x\in V} \dir((\alpha_e)_{e\in E,\ \underline{e}=x}).$$ Note that this distribution does not satisfy the usual uniform ellipticity condition: there is no positive constant bounding $\bP$-almost surely the transition probabilities $\omega_e$ from above. 

In the case of $\mathbb{Z}^d$, we always consider translation invariant distributions of environments, hence the parameters are identical at each vertex and we only need to be given a $2d$-uplet $(\alpha_e)_{e\in\mathcal{V}}$ where $\mathcal{V}=\ensemble{e\in\Z^d}{|e|=1}$. 

For any environment $\omega\in\Omega$, and $x\in V$, we denote by $P_{x,\omega}$ the law of the Markov chain starting at $x$ with transition probabilities given by $\omega$, and by $(X_n)_{n\geq0}$ the canonical process on $V$. The {\bf annealed law starting at $x\in V$} is then the following averaged distribution on random walks on $G$: $$P_x(\cdot)=\int P_{x,\omega}(\cdot) \bP(d\omega) = \bE[P_{x,\omega}(\cdot)].$$ 

\label{par:simplification}The associativity property of the Dirichlet distribution allows to reduce graphs with multiple edges between two vertices to simple-edged ones. Consider indeed the graph $G'=(V,E')$ deduced from $G$ by replacing multiple (oriented) edges by single ones bearing weights $\alpha_{e'}$ equal to the sum of the weights of the edges they replace. On $G$, the quenched laws $P_{x_0,\omega}$ depend only on the sums $\sum_{\underline{e}=x,\,\overline{e}=y} \omega_e$ for $x,y\in V$ and, thanks to associativity, the joint law under $\bP$ of these sums is the Dirichlet distribution relative to the graph $G'$. Hence the annealed laws on $G$ and $G'$ are the same and, for the problems we are concerned with, we may use $G'$ instead of $G$. We may therefore assume that, unless otherwise explicitly stated, \emph{all graphs to be considered do not have multiple edges}. We then denote by $(x,y)$ \emph{the} edge from $x$ to $y$, and we usually write $\omega(x,y)$ instead of $\omega_{(x,y)}$. 

We will need the following stopping times: $T_A=\inf\ensemble{n\geq 1}{(X_{n-1},X_n)\notin A}$ for $A\subset E$, $T_U=\inf\ensemble{n\geq 0}{X_n\notin U}$ for $U\subset V$ and, for every vertex $x$, $H_x=\inf\ensemble{n\geq 0}{X_n=x}$ and $\widetilde{H_x}=\inf\ensemble{n\geq 1}{X_n=x}$. 

If the random variable $N_y$ denotes the number of visits of $(X_n)_{n\geq0}$ at site $y$, then the {\bf Green function} $G^\omega$ of the random walk in the environment $\omega$ is given by: $$\mbox{for all $x,y\in V$, }G^\omega(x,y)=E_{x,\omega}[N_y]=\sum_{n\geq 0}P_{x,\omega}(X_n=y).$$

Due to the assumptions (i) and (ii), $G^\omega(x,y)$ is $\bP$-almost surely finite for all $x,y\in V$. The question we are concerned with is the integrability of these functions under $\bP$, according to the value of $\vec{\alpha}$. 

\subsection{Integrability conditions}

The main quantity involved in our conditions is the sum of the coefficients $\alpha_e$ over the edges $e$ exiting some set. Let us give a few last notations about sets of edges. For every subset $A$ of $E$, define: 
\eqd
\underline{A} & = & \ensemble{\underline{e}}{e\in A}\subset V,\\
\overline{A} & = & \ensemble{\overline{e}}{e\in A}\subset V,\\
\overline{\underline{A}} & = & \ensemble{\underline{e}}{e\in A}\cup\ensemble{\overline{e}}{e\in A}\subset V,\\
\partial_E A & = & \ensemble{e\in E\setminus{A}}{\underline{e}\in\underline{A}}\subset E,
\eqf
and the sum of the coefficients of the edges "exiting $A$":
$$\beta_A=\sum_{e\in \partial_E A}\alpha_e.$$
$A$ is said to be {\bf strongly connected} if, for all $x,y\in\overline{\underline{A}}$, $x$ is {\bf connected} to $y$ in $A$, i.e. there is an (oriented) path from $x$ to $y$ through edges in $A$. 

Our main result is the following:

\begin{thm} \label{thm:main}
Let $G=(V\cup\{\partial\},E)$ be a finite directed graph and $\vec{\alpha}=(\alpha_e)_{e\in E}$ a family of positive real numbers. We denote by $\bP$ the Dirichlet distribution with parameter $\vec{\alpha}$. Let $o\in V$. For every $s>0$, the following statements are equivalent: 
\begin{enumerate}
	\item $\bE[G^\omega(o,o)^s]<\infty$; 
	\item for every strongly connected subset $A$ of $E$ such that $o\in\underline{A}$, $\beta_A>s$. 
\end{enumerate}
\end{thm}

\emph{Undirected} graphs are directed graphs where edges come in pair: if $(x,y)\in E$, then $(y,x)\in E$ as well. In this case, the previous result translates into a statement on subsets of $V$. For any $S\subset V$, we denote by $\beta_S$ the sum of the coefficients of the edges "exiting $S$": 
$$\beta_S=\sum_{\underline{e}\in S,\ \overline{e}\notin S} \alpha_e.$$
For any strongly connected subset $A$ of $E$, if $S=\underline{A}$, we have $\beta_S\leq\beta_A$ and equality holds if $A$ contains every edge in $E$ linking vertices of $\underline{A}$ and if the graph contains no loop (i.e. no edge exiting from and heading to the same vertex). This remark yields: 

\begin{thm}
Let $G=(V\cup\{\partial\},E)$ be a finite undirected graph without loop and $(\alpha_e)_{e\in E}$ a family of positive real numbers. We denote by $\bP$ the corresponding Dirichlet distribution. Let $o\in V$. For every $s>0$, the following statements are equivalent: 
\begin{enumerate}
	\item $\bE[G^\omega(o,o)^s]<\infty$;
	\item for all connected subsets $S$ of $V$ such that $\{o\}\subsetneq S$, $\beta_S>s$. 
\end{enumerate}
\end{thm}

In particular, we get the case of i.i.d. environments in $\mathbb{Z}^d$. Noticing that the non-empty subsets $A$ of edges of $\mathbb{Z}^d$ with smallest "exit sum" $\beta_A$ are the sets containing one single edge, the result may be restated like this: 

\begin{thm} {\label{thm:zd}}
Let $\vec{\alpha}=(\alpha_e)_{e\in\mathcal{V}}$ be a family of positive real numbers. We denote by $\bP$ the translation invariant Dirichlet distribution on environments on $\mathbb{Z}^d$ deduced from $\vec{\alpha}$. Let $U$ be a finite subset of $\mathbb{Z}^d$. Set $\Sigma=\sum_{e\in\mathcal{V}}\alpha_e$. Then for every $s>0$, the following assertions are equivalent: 
\begin{enumerate}
	\item for all $x\in U$, $\bE[G^\omega(x,x)^s]<\infty$;
	\item every edge $e\in\mathcal{V}$ such that there is $x\in U$ with $x+e\in U$ satisfies: $2\Sigma > \alpha_e + \alpha_{-e} + s$. 
\end{enumerate}
\end{thm}

Assuming the hypothesis of theorem~\ref{thm:main} relatively to all vertices instead of only one provides information about exit times: 

\begin{coroll} {\label{cor:oriented}}
Let $G=(V\cup\{\partial\},E)$ be a finite directed strongly connected graph and $(\alpha_e)_{e\in E}$ a family of positive real numbers. For every $s>0$, the following properties are equivalent: 
\begin{enumerate}
	\item for every vertex $x$, $\bE[E_{x,\omega}[T_V]^s]<\infty$; 
	\item for every vertex $x$, $\bE[G^\omega(x,x)^s]<\infty$; 
	\item every non-empty strongly connected subset $A$ of $E$ satisfies $\beta_A>s$;
	\item there is a vertex $x$ such that $\bE[E_{x,\omega}[T_V]^s]<\infty$.
\end{enumerate}
\end{coroll}

And in the undirected case: 

\begin{coroll}
Let $G=(V\cup\{\partial\},E)$ a finite connected undirected graph without loop, and $(\alpha_e)_{e\in E}$ a family of positive real numbers. For every $s>0$, the following properties are equivalent:
\begin{enumerate}
	\item for every vertex $x$, $\bE[E_{x,\omega}[T_V]^s]<\infty$; 
	\item for every vertex $x$, $\bE[G^\omega(x,x)^s]<\infty$; 
	\item every connected subset $S$ of $V$ of cardinality $\geq 2$ satisfies $\beta_S>s$;
	\item there is a vertex $x$ such that $\bE[E_{x,\omega}[T_V]^s]<\infty$.
\end{enumerate}
\end{coroll}

\subsection*{Ballisticity criterion}

We now consider the case of random walks in i.i.d. Dirichlet environment on $\Z^d$, $d\geq 1$. 

Let $(e_1,\ldots,e_d)$ denote the canonical basis of $\Z^d$, and $\mathcal{V}=\ensemble{e\in\Z^d}{|e|=1}$. Let $(\alpha_e)_{e\in\mathcal{V}}$ be positive numbers. We will write either $\alpha_i$ or $\alpha_{e_i}$, and $\alpha_{-i}$ or $\alpha_{-e_i}$, $i=1,\ldots,d$. 

Enriquez and Sabot proved ballistic behaviour of the random walk in Dirichlet environment as soon as $\max_{1\leq i\leq d} |\alpha_i-\alpha_{-i}|>1$. Our improvement replaces $l^\infty$-norm by $l^1$-norm:

\begin{thm} \label{thm:bali}
If $\sum\limits_{i=1}^d |\alpha_i-\alpha_{-i}|>1$, then there exists $v\neq 0$ such that, $P_0$-a.s., $\displaystyle\frac{X_n}{n}\to_n v,$
and the following bound holds: 
$$\left| v - \frac{\Sigma}{\Sigma-1}d_m\right|_1 \leq \frac{1}{\Sigma-1},$$
where $\Sigma=\sum_{e\in\mathcal{V}}\alpha_e$, $d_m=\sum_{i=1}^d \frac{\alpha_i-\alpha_{-i}}{\Sigma} e_i$ is the drift in the averaged environment, and $|X|_1=\sum_{i=1}^d |X\cdot e_i|$ for any $X\in\R^d$. 
\end{thm}

\section{Proof of the main result} \label{sec:proof}

Let us first give a few words on the proof of theorem~\ref{thm:main}. Proving this integrability condition amounts to bound the tail probability $\bP(G^\omega(o,o)>t)$ from below and above. 

In order to get the lower bound, we exhibit an event consisting of environments preventing the walk from exiting easily from a given subset; this forces the mean exit time of this subset to be large. However, getting a large number of returns to the starting vertex $o$ requires an additional trick: one needs to control from below the probability of some paths leading back to $o$. The key fact here is the basic remark that at, at each vertex, there is at least one exiting edge whose transition probability is greater than the inverse number of neighbours of that vertex. By restricting the probability space to an event where, at each vertex, this (random) edge is fixed, we thus compensate for the non uniform ellipticity of $\bP$ and get the required bound. 

The upper bound is more elaborate. First, using a method based on the above mentioned key fact, we define a random subset $C(\omega)$ of edges such that $o\in\underline{C(\omega)}$, either $o$ or $\partial$ belongs to $\overline{C(\omega)}$, and the random walk can easily connect points through $C(\omega)$ : there is a positive constant $c$ (depending only on $G$) such that, for all distinct points $x\in\underline{C(\omega)}$ and $y\in\overline{C(\omega)}$, we have \begin{equation}P_{x,\omega}(H_y<\widetilde{H}_x\wedge T_{C(\omega)})>c.\label{eqn:minor}\end{equation} 
Note that if $\partial\in\overline{C(\omega)}$, then $G^\omega(o,o)=1/P_{o,\omega}(H_\partial<\widetilde{H}_o)\leq 1/c$ and the desired tail probability is trivial. Suppose now to the contrary that $o\in\overline{C(\omega)}$. Bound~(\ref{eqn:minor}) shows that a visit to any point of $\underline{C(\omega)}$ is likely to be followed up with a visit to $o$. Hence $G^\omega(o,o)$ is on the order of the average total time spent in $C(\omega)$ before $H_\partial$. This total time decomposes into the sum of the time spent in $C(\omega)$ between excursions out of $C(\omega)$. On the other hand, bound~(\ref{eqn:minor}) shows as well that the average exit time out of $C(\omega)$ (from any vertex of $\underline{C(\omega)}$) is on the order of $1/\Sigma$ where $\Sigma = \sum_{e\in\partial_E C(\omega)}\omega_e$. It should then be no surprise that $G^\omega(o,o)$ can be bounded from above by $G^{\widetilde{\omega}}(\ot,\ot)/\Sigma$ (up to a constant factor), where we introduced the quotient graph $\widetilde{G}$ of $G$ obtained by contracting the edges of $C(\omega)$ to a new vertex $\widetilde{o}$, and $\widetilde{\omega}$ is a suitable environment on $\widetilde{G}$. This inequality almost reduces the problem to a smaller graph. Using properties of Dirichlet distributions, we are able to replace $\widetilde{\omega}$ by a Dirichlet environment (lemma~\ref{lem:restriction}) and finally to carry out the induction argument. Note that the proof below is written in terms of the probability $P_{o,\omega}(H_\partial<\widetilde{H}_o)$, which is just $1/G^\omega(o,o)$. 

\subsection*{First implication (lower bound)}

We suppose there exists a strongly connected subset $A$ of $E$ such that $o\in\underline{A}$ and $\beta_A\leq 1$. We shall prove the stronger statement that $\bE[G^\omega_A(o,o)]=\infty$ where $G^\omega_A$ is the Green function of the random walk in the environment $\omega$ killed when exiting $A$. 

Let $\eps>0$. Define the event $\mathcal{E}_\eps=\{\forall e\in\partial_E A,\ \sum_{e\in\partial_E A,\ \underline{e}=x} \omega_e\leq\eps\}$. On $\mathcal{E}_\eps$, one has: 
$$E_{o,\omega}[T_A]\geq\frac{1}{\eps}.$$
Indeed, by Markov property, for all $n\in\mathbb{N}^*$, 
$$P_{o,\omega}(T_A> n)= E_{o,\omega}[T_A> n-1, P_{X_{n-1},\omega}(T_A>1)]\geq P_{o,\omega}(T_A>n-1)\min_{x\in\overline{A}} P_{x,\omega}(T_A> 1) $$
and, on $\Er_\eps$, for all $x\in\overline{A}=\underline{A}$, $P_{x,\omega}(T_A> 1)\geq 1-\eps$, hence: 
$$P_{o,\omega}(T_A> n)\geq P_{o,\omega}(T_A> n-1)(1-\eps)\geq\cdots\geq P_{o,\omega}(T_A>0)(1-\eps)^n=(1-\eps)^n.$$
Therefore, $$E_o[T_A]=\bE[E_{o,\omega}[T_A]]=\int_0^\infty \bP(E_{o,\omega}[T_A]\geq t) dt\geq \int_0^\infty \bP(\Er_{1/t})dt.$$
Finally, we have $\bP(\Er_\eps)\sim_{\eps\to0}C\eps^{\beta_A}$ where $C$ is a positive constant (as can be seen from definition~($*$) of Dirichlet distribution), and $\beta_A\leq1$, so that $E_o[T_A]=\infty$. As $T_A=\sum_{x\in\underline{A}}N_{A,x}$, where $N_{A,x}$ is the number of visits at $x$ before $T_A$, there exists a vertex $x\in\underline{A}$ such that $\infty=E_o[N_{A,x}]=\bE[G^\omega_A(o,x)]$. The inequality $G^\omega_A(o,x)\leq G^\omega_A(x,x)$ then yields $\bE[G^\omega_A(x,x)]=\infty$. 

In order to get the result on $G^\omega_A(o,o)$, we have to refine this proof by considering an event where there is a path from $x$ to $o$ whose transition probability is uniformly bounded from below. 

To $\bP$-almost every environment $\omega$, one can associate the subset of edges $\vec{G}(\omega)$ containing only one edge $e$ exiting from every vertex $\neq\partial$, namely the one maximizing $\omega(e)$. Then, if $e\in\vec{G}(\omega)$: 
$$\omega_e\geq\frac{1}{n_{\underline{e}}},$$
where $n_x$ is the number of neighbours of $x\in V$. In particular, there is a positive constant $\kappa$ depending only on $G$ such that, if $x$ is connected to $y$ through a (simple) path $\pi$ in $\vec{G}(\omega)$ then $P_{x,\omega}(\pi)\geq\kappa$. 

The strongly connected subset $A$ of $E$ possesses at least one spanning tree $T$ oriented towards $o$. Let us denote by $\mathcal{F}$ the event $\{\vec{G}(\omega)=T\}$ (for $\omega\in\mathcal{F}$, every vertex of $\underline{A}$ is then connected to $o$ in $\vec{G}(\omega)$). One still has $\bP(\Er_\eps\cap\mathcal{F})\geq\bP(\Er_\eps\cap\{\forall e\in T,\omega_e>1/2\})\sim_{\eps\to0} C\eps^{\beta_A}$, where $C$ is a positive constant. Then, like previously, because $\beta_A\leq 1$: 
$$\bE[E_{o,\omega}[T_A],\mathcal{F}]=\int_0^\infty\bP(E_{o,\omega}[T_A]\geq t,\mathcal{F})dt\geq\int_0^\infty\bP(\Er_{1/t}\cap\mathcal{F})dt=+\infty,$$
and subsequently there exists $x\in\underline{A}$ such that $\bE[G^\omega_A(o,x),\mathcal{F}]=\infty$, hence $\bE[G^\omega_A(x,x),\mathcal{F}]=\infty$. Now, there is an integer $l$ and a real number $\kappa>0$ such that, if $\omega\in\mathcal{F}$, $P_{x,\omega}(X_l=o)\geq\kappa$, which implies, on $\mathcal{F}$, thanks to Markov property: 
$$G^\omega_A(x,x)\leq\frac{1}{\kappa} G^\omega_A(x,o)\leq\frac{1}{\kappa} G^\omega_A(o,o).$$
(indeed, $G^\omega_A(x,x)=\sum_{k\geq0} P_{x,\omega}(X_k=x,H_\partial>k)\leq\sum_{k\geq0}\frac{1}{\kappa}P_{x,\omega}(X_{k+l}=o,H_\partial>k+l)\leq\frac{1}{\kappa}G^\omega_A(x,o)$) Therefore we get: 
$$\bE[G^\omega_A(x,x),\mathcal{F}]\leq\frac{1}{\kappa}\bE[G^\omega_A(o,o),\mathcal{F}]\leq\frac{1}{\kappa}\bE[G^\omega_A(o,o)],$$
and finally $\bE[G^\omega_A(o,o)]=\infty$. 

\subsection*{Converse implication (upper bound)}

The proof of the other implication procedes by induction on the number of edges of the graph, through quotienting by an appropriate subset of edges. 

\begin{definition}
If $A$ is a strongly connected subset of edges of a graph $G=(V,E,\mathit{head},\mathit{tail})$, the quotient graph of $G$ obtained by contracting $A$ to $\widetilde{a}$ is the graph $\widetilde{G}$ deduced from $G$ by deleting the edges of $A$, replacing all the vertices of $\underline{A}$ by one new vertex $\widetilde{a}$, and modifying the endpoints of the edges of $E\setminus A$ accordingly. Thus the set of edges of $\widetilde{G}$ is naturally in bijection with $E\setminus A$ and can be thought of as a subset of $E$. 
\end{definition}

In other words, $\widetilde{G}=(\widetilde{V},\widetilde{E},\widetilde{\mathit{head}},\widetilde{\mathit{tail}})$ where $\widetilde{V}=(V\setminus\underline{A})\cup\{\widetilde{a}\}$ ($\widetilde{a}$ being a new vertex), $\widetilde{E}=E\setminus A$ and, if $\pi$ denotes the projection from $V$ to $\widetilde{V}$ (i.e. $\pi_{|V\setminus\underline{A}}=\mathrm{id}$ and $\pi(x)=\widetilde{a}$ if $x\in\underline{A}$), $\widetilde{\mathit{head}} = \pi\circ\mathit{head}$ and $\widetilde{\mathit{tail}} = \pi\circ\mathit{tail}$ on $\widetilde{E}=E\setminus A$. Notice that this quotient may well introduce multiple edges. 

Let us first describe the construction of the subset $C(\omega)$ about to play the role of $A$ in the definition. This subset of "easily visited edges" extends the idea underlying the definition of $\vec{G}(\omega)$ in the previous proof. 

We define inductively a finite sequence $e_1=(x_1,y_1),\ldots,e_n=(x_n,y_n)$ of edges in the following way: letting $y_0=o$, if $e_1,\ldots, e_{k-1}$ have been defined, then $e_k$ is the edge in $E$ which maximizes the exit distribution out of $C_k=\{e_1,\ldots,e_{k-1}\}$ starting at $y_{k-1}$ :
$$e\mapsto P_{y_{k-1},\omega}((X_{T_{C_k}-1},X_{T_{C_k}})=e),$$ 
and $n$ is the least index $\geq 1$ such that $y_n\in\{o,\partial\}$. In words, the edge $e_k$ is, among the edges exiting the set $C_k(\omega)$ of already visited edges, the one maximizing the probability for a random walk starting at $y_{k-1}$ to exit $C_k(\omega)$ through it; and the construction ends as soon as an edge $e_k$ heads at $o$ or $\partial$. Notice that $C_1=\varnothing$, hence $T_{C_1}=1$, and more generally if $y_{k-1}\notin\{x_1,\ldots,x_{k-1}\}$, then $e_k$ maximizes in fact $e\mapsto\omega_e$ among the edges exiting $y_{k-1}$: $e_k$ is the edge of $\vec{G}(\omega)$ originating at $y_{k-1}$. The assumption that each vertex is connected to $\partial$ guarantees the existence of an exit edge out of $C_k(\omega)$ for $k\leq n$, and the finiteness of $G$ ensures that $n$ exists: the procedure ends. We set: $$C(\omega)=C_{n+1}=\{e_1,\ldots,e_n\}.$$ Note that the maximizing edges, and thus $C(\omega)$, are well defined up to a zero Lebesgue measure set. 

The support of the distribution of $\omega\mapsto C(\omega)$ writes as a disjoint union $\mathcal{C}=\mathcal{C}_o\cup\mathcal{C}_\partial$ depending on whether $o$ or $\partial$ belongs to $\overline{C(\omega)}$. For any $C\in\mathcal{C}$, we let $\Er_C$ be the event $\{C(\omega)=C\}$. On such an event, we can get uniform lower bounds on some probabilities, as if a uniform ellipticity property held: 

\begin{prop}\label{prop:onC} There exists a constant $c>0$ such that, for all $C\in\mathcal{C}$ , for all $x\in\overline{C}\setminus\{o\}$, for all $\omega\in\Er_C$, $$P_{o,\omega}(H_x<\Ht_o\wedge T_C)\geq c.$$
\end{prop}

\begin{proof}
Let $\omega\in\Er_C$. For $k=1,\ldots,n$, due to the choice of $e_k$ as a maximizer over $E$ (or $\partial_E C_k$), we have:
$$P_{y_{k-1},\omega}((X_{T_{C_k}-1},X_{T_{C_k}})=e_k)\geq\kappa=\frac{1}{|E|}.$$
As $y_k\neq o$ as soon as $k<n$, we deduce, for such $k$, and for $k=n$ if $y_n=\partial$: 
$$P_{y_{k-1},\omega}(H_{y_k}<\widetilde{H_o}\wedge T_C)\geq P_{y_{k-1},\omega}(X_{T_{C_k}}=y_k)\geq\kappa.$$
Then, by Markov property, for any $x\in\overline{C}=\{y_1,\ldots,y_n\}$, if $x\neq o$, 
$$P_{o,\omega}(H_x< \widetilde{H_o}\wedge T_C)\geq \kappa^n\geq c = \kappa^{|E|}.$$
\end{proof}

Let us now prove the upper bound itself. We prove the following property by induction on $n\geq 1$: 

\noindent\textsc{Induction hypothesis} (at rank $n$) - \begin{sffamily}
For every directed graph $G=(V\cup\{\partial\},E)$ possessing at most $n$ edges (and such that every vertex is connected to $\partial$), for every parameter family $(\alpha_e)_{e\in E}$ and every vertex $o\in V$, there exist real numbers $C,r>0$ such that, for small $\eps>0$, 
$$\bP(\qP{o}(H_\partial<\Ht_o)\leq\eps)\leq C\eps^\beta(-\ln\eps)^r,$$
where $\beta=\min \ensemble{\beta_A}{A \mbox{ is a strongly connected subset of $E$ and }o\in\underline{A}}$. 
\end{sffamily}

Initialization: if $|E|=1$ (and $o$ is connected to $\partial$), the only edge links $o$ to $\partial$, so that $\qP{o}(H_\partial<\Ht_o)=1$, and the property is true (with any positive $\beta$). 

Let $n\geq 1$. We suppose the induction hypothesis to be true at rank $n$. Let $G=(V\cup\{\partial\},E)$ be a graph with $n+1$ edges (such that every vertex is connected to $\partial$), let $o\in V$, and $(\alpha_e)_{e\in E}$ be positive real numbers. We apply to $G$ the construction of $\omega\mapsto C(\omega)$ described before and use the notations thereof. 

Because of the finiteness of $\mathcal{C}(\subset\mathcal{P}(E))$, it will be sufficient to prove the upper bound separately on each of the events $\Er_C$. Let $C\in\mathcal{C}$. 

If $C\in\mathcal{C}_\partial$, then $\partial\in \overline{C}$, and the proposition above provides $c>0$ such that, on $\Er_C$, $P_{o,\omega}(H_\partial<\Ht_o)\geq c$ hence, for small $\eps>0$, $$\bP(\qP{o}(H_\partial<\Ht_o)\leq\eps,\ \Er_C)=0.$$

Therefore we may assume $C\in\mathcal{C}_o$. Then $C$ is a strongly connected subset of $E$ (due to the construction method, for every pair of vertices of $\overline{C}$, the first encountered one can be connected  in $C$ to the second one; and $o$ is encountered both first and last), and proposition~\ref{prop:onC} provides a constant $c>0$ such that for any $x\in \overline{C}\setminus\{o\}$ and $\omega\in\Er_C$, $$P_{o,\omega}(H_x<H_\partial\wedge \Ht_o)\geq c.$$

We consider the quotient graph $\widetilde{G}=(\widetilde{V}\!=\!(V\setminus\underline{C})\cup\{\partial,\ot\},\widetilde{E}\!=\!E\setminus C,\widetilde{\mathit{head}},\widetilde{\mathit{tail}})$ obtained by contracting $C$ to a new vertex $\ot$. Because $\widetilde{E}$ is a subset of $E$, we may endow the set $\widetilde{\Omega}$ of environments on $\widetilde{G}$ with the Dirichlet distribution of parameter $(\alpha_e)_{e\in\widetilde{E}}$, again denoted by $\bP$. We may as well introduce another distribution, namely the law under $\bP$ of the \emph{quotient environment} $\widetilde{\omega}$ defined the following way: for every edge $e\in\widetilde{E}$, if $e\notin\partial_E C$ (i.e. if $\widetilde{\mathit{tail}}(e)\neq\ot$), then $\widetilde{\omega}_e=\omega_e$, where $\omega$ is the canonical random variable on $\widetilde{\Omega}$, and if $e\in\partial_E C$, then: $$\widetilde{\omega}_e=\frac{\omega_e}{\Sigma},$$
where $\Sigma=\sum_{e\in\partial_E C}\omega_e$. In the following, we shall sometimes write an exponent $G$ or $\widetilde{G}$ on the probability $\qP{x}$ to indicate which graph we consider. 

The edges in $C$ do not appear in $\widetilde{G}$ anymore. In particular, $\widetilde{G}$ has strictly less than $n$ edges. In order to apply induction with respect to the point $\ot$, it suffices to check that each vertex is connected to $\partial$, which results easily from the similar property for $G$. As the induction hypothesis applies to graphs with simple edges, we denote by $\widetilde{\beta}$ the exponent "$\beta$" in the induction hypothesis corresponding now to the graph $\widetilde{G}$ once its multiple edges have been simplified (see page~\pageref{par:simplification}) and to $\ot$. Then, using the induction hypothesis, we have, for small $\eps>0$: 
\begin{equation}\label{eqn:induc}
\bP(P^{\widetilde{G}}_{\ot,\omega}(H_\partial<\Ht_\ot)\leq\eps)\leq C\eps^{\widetilde{\beta}}(-\ln\eps)^r,
\end{equation}
where $C,r>0$, and the left-hand side may equivalently refer to the graph $\widetilde{G}$ or to its simple-edged version, as explained earlier. 

We need to come back from $\widetilde{G}$ to $G$. First it will be easier to do so with the environment $\widetilde{\omega}$. Notice that, from $o$, one way for the walk to reach $\partial$ without coming back to $o$ consists in exiting $C$ without coming back to $o$ and then reaching $\partial$ without coming back to $\underline{C}$. Thus we have, $\bP^{(G)}$-a.s. on $\mathcal{E}_C$: 
\eqd
P_{o,\omega}(H_\partial<\Ht_\ot)
	& \geq & \sum_{x\in\underline{C}} P_{o,\omega}(H_x<\Ht_o\wedge T_C, H_\partial< H_x+\Ht_{\underline{C}}\circ\Theta_{H_x})\\
	& = & \sum_{x\in\underline{C}} P_{o,\omega}(H_x<\Ht_o\wedge T_C)P_{x,\omega}(H_\partial<\Ht_{\underline{C}})\\
	& \geq & c \sum_{x\in\underline{C}} P_{x,\omega}(H_\partial<\Ht_{\underline{C}})\\
	& = & c\Sigma\cdot P_{\ot,\widetilde{\omega}}(H_\partial<\Ht_\ot)
\eqf
where the last equality comes from the definition of the quotient : both quantities correspond to the same set of paths viewed in $G$ and in $\widetilde{G}$, and, for all $x\in\underline{C}$, $P_{x,\omega}$-almost every path belonging to the event $\{H_\partial<\Ht_{\underline{C}}\}$ contains exactly one edge exiting from $\underline{C}$ so that the renormalization by $\Sigma$ appears exactly once when considering $\widetilde{\omega}$. 

Thus, for some $c'>0$, we have:
$$\bP(P^G_{o,\omega}(H_\partial<\Ht_o)\leq\eps,\ \mathcal{E}_C)\leq\bP(\Sigma\cdot P^{\widetilde{G}}_{\ot,\widetilde{\omega}}(H_\partial<\Ht_\ot)\leq c'\eps,\ \mathcal{E}_C).$$

Remark that $\widetilde{\omega}$ does not follow a Dirichlet distribution because of the renormalization. We can however reduce to the Dirichlet situation and thus procede to induction. This is the aim of the following lemma, inspired by the restriction property of section~\ref{sec:dir}. For readibility and tractability reasons, we only state and prove it in the case of two Dirichlet random variables, though the generalization is fairly straightforward: 

\begin{lemme}\label{lem:restriction}
Let $(p_1,\ldots,p_{k+1})$ and $(p'_1,\ldots,p'_{l+1})$ be random variables following, under the probability $P$, Dirichlet laws of respective parameters $\vec{\alpha}$ and $\vec{\alpha}'$. We set $\Sigma=p_1+\ldots+p_k+p'_1+\ldots+p'_l$. Then there exists $C>0$ such that, for every positive measurable function $f:\R^{k+l+1}\to\R$, 
$$E\left[f\left(\Sigma,\frac{p_1}{\Sigma},\ldots,\frac{p_k}{\Sigma},\frac{p'_1}{\Sigma},\ldots,\frac{p'_l}{\Sigma}\right)\right]
	\leq C\cdot \widetilde{E}\left[f(\widetilde{\Sigma},\widetilde{p}_1,\ldots,\widetilde{p}_k,\widetilde{p'}_1,\ldots,\widetilde{p'}_l)\right],$$
where, under the probability $\widetilde{P}$, $(\widetilde{p}_1,\ldots,\widetilde{p}_k,\widetilde{p'}_1,\ldots,\widetilde{p'}_l)$ is sampled from a Dirichlet distribution of parameter $(\alpha_1,\ldots,\alpha_k,\alpha'_1,\ldots,\alpha'_l)$, $\widetilde{\Sigma}$ is bounded and satisfies $\widetilde{P}(\widetilde{\Sigma}<\eps)\leq C'\eps^{\alpha_1+\cdots+\alpha_k+\alpha'_1+\cdots+\alpha'_l}$ for every $\eps>0$, and these two variables are independent.  
\end{lemme}

\begin{proof}
We set $\beta=\alpha_{k+1}$ and $\beta'=\alpha'_{l+1}$. 
Writing the index $(\cdot)_i$ instead of $(\cdot)_{1\leq i\leq k}$ and the same way with $j$ and $l$, the left-hand side of the statement equals: 
$$\int_{\scriptsize\begin{array}{c}\{\sum\limits_i x_i\leq 1,\\ \sum\limits_j y_j\leq 1\}\end{array}}f\left(\sum\limits_i x_i+\sum\limits_j y_j,\left(\frac{x_i}{\sum_i x_i+\sum_j y_j}\right)_i,\left(\frac{y_j}{\sum_i x_i+\sum_j y_j}\right)_j\right) \phi((x_i)_i,(y_j)_j)\, \prod_i {dx_i}\prod_j{dy_j},$$
where for some positive $c_0$, $\phi((x_i)_i,(y_j)_j)=c_0\left(\prod\limits_i x_i^{\alpha_i-1}\right)(1-\sum\limits_i x_i)^\beta\left(\prod\limits_j y_j^{\alpha'_j-1}\right)(1-\sum\limits_j y_j)^{\beta'}$. We successively procede to the following changes of variable : $x_1\mapsto u=\sum_i x_i +\sum_j y_j$, then $x_i\mapsto \widetilde{x}_i=\frac{x_i}{u}$ for every $i\neq 1$, and $y_j\mapsto \widetilde{y}_j=\frac{y_j}{u}$ for every $j$. The previous integral becomes : 
$$\int_{\scriptsize\left\{\begin{array}{c}\sum_{i\neq 1} \widetilde{x}_i+\sum_j \widetilde{y}_j\leq 1,\\1-\frac{1}{u}\leq\sum_j \widetilde{y}_j\leq\frac{1}{u}\end{array}\right\}}f\left(u,1-\sum_{i\neq 1} \widetilde{x}_i + \sum_j \widetilde{y}_j,\left(\widetilde{x}_i\right)_{i\neq 1},\left(\widetilde{y}_j\right)\right) \psi(u,(\widetilde{x}_i)_{i\neq1},(\widetilde{y}_j)_j)\,du\prod_{i\neq 1} {d\widetilde{x}_i}\prod_j{d\widetilde{y}_j},$$
where : $$\textstyle\psi(u,(x_i)_{i\neq1},(y_j)_j)=c_0 u^{\sum\limits_i \alpha_i+\sum\limits_j\alpha'_j-1}(1-\sum\limits_{i\neq1}\widetilde{x}_i-\sum\limits_j \widetilde{y}_j)^{\alpha_1-1} \prod\limits_{i\neq1} \widetilde{x}_i^{\alpha_i-1}(1-u(1-\sum\limits_j \widetilde{y}_j))^\beta \prod\limits_j \widetilde{y}_j^{\alpha'_j-1}(1-u \sum\limits_j \widetilde{y}_j)^{\beta'}.$$
Bounding from above by 1 the last two factors of $\psi$ where $u$ appears, we get that the last quantity is less than:
$$ \int_{\scriptsize\left\{\begin{array}{c}\sum_{i\neq 1} \widetilde{x}_i+\sum_j \widetilde{y}_j\leq 1,\\u\leq 2\end{array}\right\}}f\left(u,1-\sum_{i\neq 1} \widetilde{x}_i + \sum_j \widetilde{y}_j,\left(\widetilde{x}_i\right)_{i\neq 1},\left(\widetilde{y}_j\right)\right)\theta(u,(\widetilde{x}_i)_{i\neq 1},(y_j)_j)\,du\prod_{i\neq 1} {d\widetilde{x}_i}\prod_j{d\widetilde{y}_j},$$
where $\theta(u,(\widetilde{x}_i)_{i\neq 1},(y_j)_j)=c_0 \left(1-\sum\limits_{i\neq1}\widetilde{x}_i-\sum\limits_j \widetilde{y}_j\right)^{\alpha_1-1} \prod\limits_{i\neq1} \widetilde{x}_i^{\alpha_i-1}\prod\limits_j \widetilde{y}_j^{\alpha'_j-1}$. 

This rewrites, for some positive $c_1$, as : 
$c_1 \widetilde{E}\left[f(\widetilde{\Sigma},\widetilde{p}_1,\ldots,\widetilde{p}_k,\widetilde{p'}_1,\ldots,\widetilde{p'}_l)\right],$ with the notations of the statement (we have here $\widetilde{P}(\widetilde{\Sigma}<\eps)=c\int_0^\eps u^{\sum_i \alpha_i+\sum_j\alpha'_j-1}du= c' \eps^{\sum_i\alpha_i+\sum_j\alpha'_j}$). 
\end{proof}

Using the inequality before the lemma we get:
\eqd
\bP(\qP{o}(H_\partial<\Ht_o)\leq\eps,\ \mathcal{E}_C)
	& \leq & \bP(\Sigma\cdot P_{\ot,\widetilde{\omega}}(H_\partial<\Ht_\ot)\leq c'\eps,\ \mathcal{E}_C)\\
	& \leq & \bP(\Sigma\cdot P_{\ot,\widetilde{\omega}}(H_\partial<\Ht_\ot)\leq c'\eps)\\
	& \leq & C'\bP(\widetilde{\Sigma}\cdot P_{\ot,\omega}(H_\partial<\Ht_\ot)\leq c'\eps),
\eqf
where, under $\bP$, $\widetilde{\Sigma}$ is a positive bounded random variable independent of $\omega$ such that $\bP(\widetilde{\Sigma}\leq\eps)\leq c\eps^{\beta_{C}}$ for all $\eps>0$. The next result will be useful to conclude: 

\begin{lemme}
If $X$ and $Y$ are independent positive bounded random variables satisfying, for some real numbers $\alpha_X,\alpha_Y,r>0$:
\begin{itemize}
	\item there exists $C>0$ such that $P(X<\eps)\leq C \eps^{\alpha_X}$ for all $\eps>0$ (or equivalently for small $\eps$);
	\item there exists $C'>0$ such that $P(Y<\eps)\leq C' \eps^{\alpha_Y}(-\ln\eps)^r$ for small $\eps>0$,
\end{itemize}
then there exists a constant $C''>0$ such that, for small $\eps>0$: 
$$P(XY\leq \eps)\leq C''\eps^{\alpha_X\wedge\alpha_Y}(-\ln\eps)^{r+1}$$
(and $r+1$ can be replaced by $r$ if $\alpha_X\neq\alpha_Y$). 

\end{lemme}

\begin{proof}
We denote by $M_X$ and $M_Y$ (deterministic) upper bounds of $X$ and $Y$. We have, for $\eps>0$:
$$P(XY\leq\eps) =  P\left(Y\leq\frac{\eps}{M_X}\right)+P\left(XY\leq\eps, Y>\frac{\eps}{M_X}\right).$$
Let $\eps_0>0$ be such that the upper bound in the statement for $Y$ is true as soon as $\eps<\eps_0$. Then, for $0<\eps<\eps_0$:
\eqd
P(XY\leq\eps, Y>\frac{\eps}{M_X})
	& = & \int_{\frac{\eps}{M_X}}^{M_Y} P\left(X\leq\frac{\eps}{y}\right)P(Y\in dy)\\
	& \leq & C \int_{\frac{\eps}{M_X}}^{M_Y} \left(\frac{\eps}{y}\right)^{\alpha_X} P(Y\in dy)\\
	& = & C \eps^{\alpha_X} E\left[{\bf 1}_{(Y\geq{\frac{\eps}{M_X}})}\frac{1}{Y^{\alpha_X}}\right]\\
	& = & C\eps^{\alpha_X} \left(\int_{\frac{\eps}{M_X}}^{M_Y} P({\frac{\eps}{M_X}}\leq Y\leq x)\frac{{\alpha_X} dx}{x^{{\alpha_X}+1}} + \frac{1}{{M_Y}^{\alpha_X}}\right)\\
	& \leq & C\eps^{\alpha_X} \left(\alpha_X C'\int_{\frac{\eps}{M_X}}^{\eps_0} x^{\alpha_Y} (-\ln x)^r \frac{dx}{x^{{\alpha_X}+1}} +C'' \right)\\
	& \leq & C\eps^{\alpha_X} \left(\alpha_X C'\int_{\frac{\eps}{M_X}}^{\eps_0} x^{{\alpha_Y}-{\alpha_X}-1} dx (-\ln{\frac{\eps}{M_X}})^r + C''\right)\\
	& \leq & C''' \eps^{{\alpha_X}\wedge{\alpha_Y}}(-\ln\eps)^{r+1}.
\eqf
Indeed, if $\alpha_Y>\alpha_X$, the integral converges as $\eps\to 0$; if $\alpha_Y=\alpha_X$, it is equivalent to $-\ln \eps$; if $\alpha_Y>\alpha_X$, the equivalent becomes $\frac{1}{\eps^{\alpha_X-\alpha_Y}}$. And the formula is checked in every case (note that $-\ln\eps>1$ for small $\eps$).
\end{proof}

In conclusion, using~(\ref{eqn:induc}), the last lemma and the inequality right before it, we get constants $c,r>0$ such that, for small $\eps>0$:
$$\bP(\qP{o}(H_\partial<\Ht_o)\leq\eps,\mathcal{E}_C)\leq c\eps^{\beta_{C}\wedge\widetilde{\beta}}(-\ln\eps)^{r+1}.$$

Notice that $\widetilde{\beta}\geq\beta$, where $\beta$ is the exponent defined in the induction hypothesis. Indeed: let $\widetilde{A}$ be a strongly connected subset of $\widetilde{E}$ such that $\ot\in\underline{\widetilde{A}}$. Set $A=\widetilde{A}\cup C\subset E$. In view of the definition of $\widetilde{E}$, every edge exiting $\widetilde{A}$ corresponds to an edge exiting $A$ and vice-versa (the only edges to be deleted by the quotienting procedure are those of $C$). Thus, $\beta_{\widetilde{A}}=\beta_A$, $o\in\underline{A}$, and $A$ is strongly connected (because so are $\widetilde{A}$ and $C$, and $\ot\in \widetilde{A}$, $o\in C$). Hence we deduce $\widetilde{\beta}\geq\beta$ as expected. 

Then $\beta_{C}\wedge\widetilde{\beta}\geq\beta_C\wedge\beta=\beta$ because $C$ is strongly connected and $o\in\underline{C}$. This concludes the induction (summing on all events $\mathcal{E}_C$, $C\in\mathcal{C}$). 

The result is then deduced from the induction property using the integrability of $t\mapsto\frac{(\ln t)^r}{t^\beta}$ in the neighbourhood of $+\infty$ as soon as $\beta>1$, and the following Markov chain identity: 
$$G^\omega(o,o)=\frac{1}{\qPo(H_\partial<\widetilde{H}_o)}.$$

\paragraph{Remark:} This proof gives the following more precise result: there exist $c,C,r>0$ such that, for large enough $t$, 
$$c\frac{1}{t^{\min_A\beta_A}}\leq\bP(G^\omega(o,o)>t)\leq C \frac{(\ln t)^r}{t^{\min_A \beta_A}},$$
where the minimum is taken over all strongly connected subsets $A$ of $E$ such that $o\in\underline{A}$. 

\subsection*{Proof of the corollary}

We prove corollary~\ref{cor:oriented}, restated here:

\begin{corollnn}
Let $G=(V\cup\{\partial\},E)$ be a finite oriented strongly connected graph and $(\alpha_e)_{e\in E}$ a family of positive real numbers. For every $s>0$, the following properties are equivalent: 
\begin{enumerate}
	\item for every vertex $x$, $\bE[E_{x,\omega}[T_V]^s]<\infty$; 
	\item for every vertex $x$, $\bE[G^\omega(x,x)^s]<\infty$; 
	\item every non-empty strongly connected subset $A$ of $E$ satisfies $\beta_A>s$;
	\item there is a vertex $x$ such that $\bE[E_{x,\omega}[T_V]^s]<\infty$.
\end{enumerate}
\end{corollnn} 

\begin{proof}
The equivalence of $(i)$ and $(ii)$ results from the inequalities below: for every $\omega\in\Omega$, $x\in V$, $s>0$, 
$$ G^\omega(x,x)^s = E_{x,\omega}[N_x]^s \leq E_{x,\omega}[T_V]^s
	 =  \left(\sum_{y\in V} P_{x,\omega}(H_y<H_\partial)G^\omega(y,y)\right)^s
	 \leq  |V|^s \sum_{y\in V} G^\omega(y,y)^s. $$

Theorem~\ref{thm:main} provides the equivalence of $(ii)$ and $(iii)$. The fact that $(i)$ implies $(iv)$ is trivial. 

Let us suppose that $(iii)$ is not satisfied: there is a strongly connected subset $A$ of $E$ such that $\beta_A\leq1$. Let $o$ be a vertex. If $o\in \underline{A}$, then $E_o[T_V]\geq \bE[G^\omega(o,o)]=\infty$; and if $o\notin \underline{A}$, there exists (thanks to strong connexity) a path $\pi$ from $o$ to some vertex $x\in \underline{A}$ which remains outside $\underline{A}$ (before $x$), and we recall that theorem~\ref{thm:main} proves $\bE[G_A^\omega(x,x)]=\infty$, hence thanks to spatial independence of the environment: 
$$E_o[T_V]\geq \bE[G^\omega(o,x)] \geq \bE[P_{o,\omega}(\pi)G_A^\omega(x,x)]=\bE[P_{o,\omega}(\pi)]\times\bE[G_A^\omega(x,x)]=\infty,$$
so that in both cases ($o\in \underline{A}$, $o\notin \underline{A}$), $E_o[T_V]=\infty$. Thus, $(iv)$ is not true. So $(iv)$ implies $(iii)$, and we are done. 
\end{proof}

\paragraph{Remark:} Under most general hypotheses, $(i)$ and $(ii)$ are still equivalent (same proof). The equivalence of $(i)$ et $(iv)$ can be shown to hold as well in the following general setting:

\begin{prop}
Let $G=(V\cup\{\partial\},E)$ be a finite strongly connected graph endowed with a probability measure $\bP$ on the set of its environments satisfying:
\begin{itemize}
	\item the transition probabilities $\omega(x,\cdot)$, $x\in V$, are independent under $\bP$;
	\item for all $e\in E$, $\bP(\omega_e>0)>0$.
\end{itemize}
If there exists $x\in V$ such that $\aE{x}[T_V]=+\infty$, then for all $y\in V$, $\aE{y}[T_V]=+\infty$. 
\end{prop}

\begin{proof}
Suppose $x\in V$ satisfies $\aE{x}[T_V]=+\infty$. We denote by $A$ a subset of $E$ satisfying $\aE{x}[T_A]=+\infty$, and being \emph{minimal} (with respect to inclusion) among the subsets of $E$ sharing this property. As $E$ is finite, the existence of such an $A$ is straightforward. 

Let $y\in \underline{A}$: there is an $e\in A$ such that $\underline{e}=y$. Let us prove $\aE{y}[T_A]=+\infty$. We have, by minimality of $A$, $\aE{x}[T_{A\setminus\{e\}}]<\infty$. Let $H_e=\inf\ensemble{n\geq 1}{(X_{n-1},X_n)=e}$. Then: 
\eqd
\aE{x}[T_A]
	& = & \aE{x}[T_A,H_e<T_A]+\aE{x}[T_A,H_e> T_A]\\
	& \leq & \aE{x}[T_A,H_e<T_A]+\aE{x}[T_{A\setminus\{e\}}],
\eqf
hence $\aE{x}[T_A,H_e<T_A]=+\infty$. Thus, using Markov property:
\eqd
+\infty
	& = & \aE{x}[T_A-T_{A\setminus\{e\}}+1,H_e<T_A]=\aE{x}[T_A-(H_e-1),H_e<T_A] \\
	& \leq & \aE{x}[T_A-(H_e-1),H_e-1<T_A] = \aE{x}[T_A\circ\Theta_{H_e-1},H_e-1<T_A]\\
	& = & \bE[\qE{x}[\qE{X_{H_e-1}}[T_A],H_e-1<T_A]] = \bE[\qE{\underline{e}}[T_A]\qP{x}(H_e-1<T_A)]\\
	& \leq & \aE{\underline{e}}[T_A],
\eqf
which gives $\aE{y}[T_A]=+\infty$ as announced. 

Let $z\in V$. If $z\in \underline{A}$, we have of course $\aE{z}[T_V]\geq \aE{z}[T_A]=+\infty$. Suppose $z\in V\setminus \underline{A}$. By strong connexity of $G$, one can find a simple path $e_1,\cdots,e_n$ from $z$ to a point $y=\overline{e_n}\in\underline{A}$ such that $\underline{e_1},\ldots,\underline{e_n}\notin\underline{A}$ (take any simple path from $z$ to any point in $\underline{A}$ and stop it just before it enters $\underline{A}$ for the first time). Then, by Markov property and using independence between the vertices in the environment: 
\eqd
\aE{z}[T_V]
	& \geq & \aE{z}[T_V, X_i=\overline{e_i}\mbox{ for }i=1,\ldots,n]\\
	& = & \bE[\omega_{e_1}\cdots\omega_{e_n}\qE{y}[T_V+n]]\\
	& \geq & \bE[\omega_{e_1}\cdots\omega_{e_n}\qE{y}[T_A+n]\\
	& = & \bE[\omega_{e_1}]\cdots\bE[\omega_{e_n}](\aE{y}[T_A]+n)
\eqf
hence $\aE{z}[T_V]=+\infty$ because the first factors are positive and the last one is infinite \emph{via} the first part of the proof. This concludes. 
\end{proof}

\section{Proof of the ballisticity criterion} \label{sec:bali}

We now consider random walks in i.i.d. Dirichlet environment on $\Z^d$, $d\geq 1$. Let $(e_1,\ldots,e_d)$ denote the canonical basis of $\Z^d$, and $\mathcal{V}=\ensemble{e\in\Z^d}{|e|=1}$. Let $(\alpha_e)_{e\in\mathcal{V}}$ be positive numbers. We will write either $\alpha_i$ or $\alpha_{e_i}$, and $\alpha_{-i}$ or $\alpha_{-e_i}$, $i=1,\ldots,d$. Let us recall the statement of theorem~\ref{thm:bali}: 

\begin{thmnn}
If $\sum\limits_{i=1}^d |\alpha_i-\alpha_{-i}|>1$, then there exists $v\neq 0$ such that, $P_0$-a.s., $\displaystyle\frac{X_n}{n}\to_n v,$
and the following bound holds: 
$$\left| v - \frac{\Sigma}{\Sigma-1}d_m\right|_1 \leq \frac{1}{\Sigma-1},$$
where $\Sigma=\sum_{e\in\mathcal{V}}\alpha_e$ and $d_m=\sum_{i=1}^d \frac{\alpha_i-\alpha_{-i}}{\Sigma} e_i$ is the drift under the averaged environment. 
\end{thmnn}

\begin{proof}
This proof relies on properties and techniques of~\cite{EnriquezSabot06}, along with two improvements: first, thanks to the previous sections, we are able to define the Kalikow random walk under weaker conditions, namely those of the statement; second, we get a finer bound on the drift of this random walk. 

Let us recall a definition. Given a finite subset $U$ of $\Z^d$ and a point $z_0\in U$ such that $\bE[G_U^\omega(z_0,z_0)]<\infty$, the {\bf Kalikow auxiliary random walk} related to $U$ and $z_0$ is the Markov chain on $U\cup\partial_V U$ (where $\partial_V U$ is the set of the vertices neighbouring $U$) given by the following transition probabilities:
$$\mbox{for all $z\in U$ and $e\in\mathcal{V}$, }\widehat{\omega}_{U,z_0}(z,z+e)=\frac{\bE[G_U^\omega(z_0,z)\omega(z,z+e)]}{\bE[G_U^\omega(z_0,z)]}$$
and $\widehat{\omega}_{U,z_0}(z,z)=1$ if $z\in\partial_V U$. For the sake of making formal computations rigorous, Enriquez and Sabot first consider the {\bf generalized Kalikow random walk}. Given an additional parameter $\delta\in(0,1)$, it is defined like the previous walk except that, in place of $G_U^\omega(z_0,z)$, we use the Green function of the random walk under the environment $\omega$ killed at rate $\delta$ and at the boundary of $U$: 
$$G_{U,\delta}^\omega(z_0,z)=E_{z_0,\omega}\left[\sum_{k=0}^{T_U}\delta^k{\bf 1}_{(X_k=z)}\right]$$ 
(and we don't need any assumption on $\bP$ anymore). 

The following identity (equation (2) of~\cite{EnriquezSabot06}) was a consequence of an integration by part formula: for all finite $U\subset\Z^d$, $z\in U$, $e\in\mathcal{V}$, $\delta\in(0,1)$,

$$\widehat{\omega}_{U,z_0,\delta}(z,z+e) = \frac{1}{\Sigma-1}\left(\alpha_e-\frac{\bE[G_{U,\delta}^\omega(z_0,z)p_{\omega,\delta}(z,z+e)]}{\bE[G_{U,\delta}^\omega(z_0,z)] ]} \right)$$

where  $p_{\omega,\delta}(z,z+e))=\omega(z,z+e)(G^\omega_{U,\delta}(z,z)-\delta G^\omega_{U,\delta}(z+e,z))$. Markov property for the killed random walk shows that, for all $z$, the components of $(p_{\omega,\delta}(z,z+e))_{e\in\mathcal{V}}$ are positive and sum up to 1: this is a probability measure. Besides, after a short computation, it can be rewritten as: 
$$p_{\omega,\delta}(z,z+e)=P_{z,\omega}(X_1=z+e|H_\partial<\widetilde{H}_z),$$
which highlights its probabilistic interpretation. Therefore the drift of the generalized Kalikow random walk at $z$ is: 
\begin{equation}\label{eqn:drift}
\widehat{d}_{U,z_0, \delta}(z)
	= \frac{1}{\Sigma-1}\left(\sum_{i=1}^d (\alpha_i-\alpha_{-i})e_i - \widetilde{d}  \right)=\frac{1}{\Sigma-1}(\Sigma d_m - \widetilde{d}), 
\end{equation}
where $\widetilde{d}$ (depending on all parameters) is the expected value of the following probability measure: $$\displaystyle \frac{\bE[G_{U,\delta}^\omega(z_0,z)p_{\omega,\delta}(z,z+\cdot)]}{\bE[G_{U,\delta}^\omega(z_0,z)]}.$$
This measure is supported by $\mathcal{V}$, hence $\widetilde{d}$ belongs to the convex hull of $\mathcal{V}$, which is the closed $|\cdot|_1$-unit ball~${\rm B}_{|\cdot|_1}$: $$|\widetilde{d}|_1\leq 1.$$ 

On the other hand, the assumption gives $\Sigma d_m\notin {\rm B}_{|\cdot|_1}$, and the convexity of ${\rm B}_{|\cdot|_1}$ provides $l\in\R_d\setminus\{0\}$ and $c>0$ (depending only on the parameters $(\alpha_e)_{e\in\mathcal{V}}$) such that, for all $X\in {\rm B}_{|\cdot|_1}$, $$\Sigma d_m\cdot l>c>X\cdot l.$$ 

Therefore, noting that our assumption implies $\Sigma>1$, we have, for every finite subset $U$ of $\Z^d$, every $z_0,z\in U$ and $\delta\in(0,1)$: $$\widehat{d}_{U,z_0, \delta}(z)\cdot l = \frac{1}{\Sigma-1}(\Sigma d_m\cdot l-\widetilde{d}\cdot l)\geq \frac{\Sigma d_m\cdot l - c}{\Sigma-1}>0.$$

It is time to remark that theorem~\ref{thm:zd} applies under our condition: the hypothesis implies $\Sigma > 1$ so that, for all $i$, $2\Sigma-\alpha_i-\alpha_{-i}>1$. This guarantees the integrability of $G^\omega_U(z_0,z)$ and allows us to make $\delta$ converge to 1 in the last inequality (monotone convergence theorem applies because $G_{U,\delta}^\omega$ increases to $G^\omega_U$ as $\delta$ increases to $1$). We get a uniform lower bound concerning the drift of Kalikow random walk: $$\widehat{d}_{U,z_0}(z)\cdot l\geq \frac{\Sigma d_m\cdot l - c}{\Sigma-1}>0.$$
 In other words, Kalikow's criterion is satisfied for finite subsets $U$ of $\Z^d$. As underlined in~\cite{EnriquezSabot06}, this is sufficient to apply Sznitman and Zerner's law of large numbers (\cite{SznitmanZerner}), hence there is a deterministic $v\neq0$ such that, $\bP$-almost surely, 
$$\frac{X_n}{n}\limite{n} v.$$ 
As for the bound on $v$, because of the similar bound for $\widehat{d}_{U,z_0,\delta}(z)$ given by the identity~(\ref{eqn:drift}) and $|\widetilde{d}|_1\leq 1$, it results from proposition 3.2 of~\cite{Sabot04} (this proposition states that $v$ is an accumulation point of the convex hull of $\ensemble{\widehat{d}_{U,z_0,\delta}(z)}{U\mbox{ finite},z_0,z\in U}$ when $\delta$ tends to 1). 
\end{proof}

\subsection*{Concluding remarks and computer simulations}

In the case of $\Z^d$, we have provided a criterion for non-zero limiting velocity. One may prove the following criterion as well, thanks to theorem~\ref{thm:zd}: 

\begin{prop}
If there exists $i\in\{1,\ldots,d\}$ such that $\alpha_i+\alpha_{-i}\geq 2\Sigma-1$, then : 
$$\aPo\AS,\ \frac{X_n}{n}\limite{n} 0.$$
\end{prop}

Indeed, the hypothesis implies that the exit time of any subset of edges containing an edge $(x,x+e_i)$, where $x\in\Z^d$, is not integrable, and the proof follows by usual arguments using the independence in the environment. 

The question remains whether one of these criterions is sharp. Actually, computer simulations let us think that neither is. We were especially able to find parameters such that exit times of all finite subsets are integrable and the random walk has seemingly zero speed (more precisely, $X_n$ looks to be on the order of $n^\kappa$ for some $0<\kappa<1$). Figure~\ref{fig:graphe} shows some results obtained with $(\alpha_1,\alpha_{-1},\alpha_2,\alpha_{-2})=(0.5,0.2,0.1,0.1)$. We performed $10^3$ numerical simulations of trajectories of random walks up to time $n_{\max{}} = 10^6$ and compared the averaged values of $y_n=X_n\cdot e_1$ with $C_\alpha n^\alpha$, where $C_\alpha$ is chosen so as to make curves coincide at $n=n_{\max{}}$. The first graph shows the average of $y_n$ and the second one the maximum over $n\in\{10^5+1,\ldots,10^6\}$ of $\left|1-\frac{y_n}{C_\alpha n^\alpha}\right|$, as $\alpha$ varies. The minimizing $\alpha$ is 0.9, corresponding to a small uniform relative error of .0044. However we could not yet prove that such an intermediary regime happens. 

\begin{figure} [b!]
\includegraphics[width=18cm]{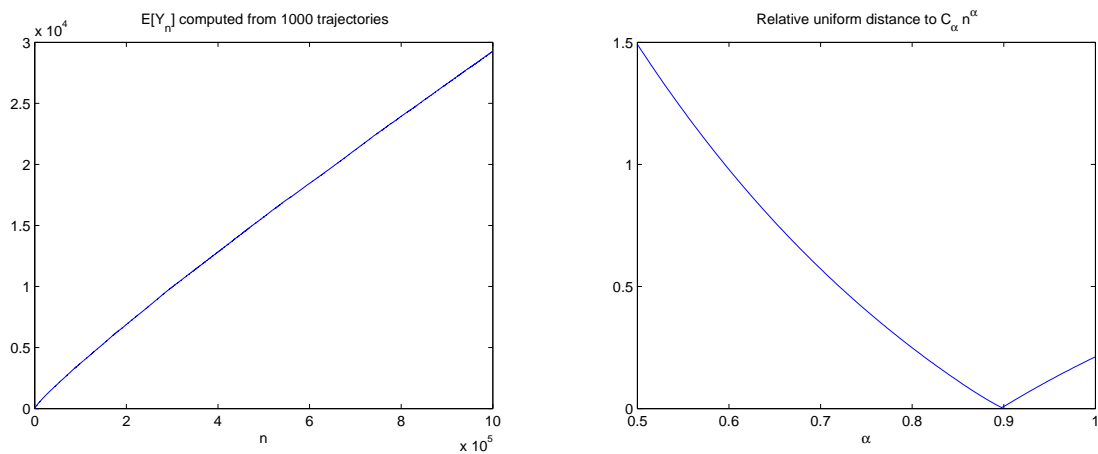} 
\caption{These plots refer to computer simulation : averages are taken over $10^3$ trajectories up to time $10^6$ (see last part of the article)}
\label{fig:graphe}
\end{figure}

\subsection*{Aknowledgements}

The author wishes to thank his advisor Christophe Sabot for suggesting the problem and for his valuable suggestions.

\end{document}